\documentclass[11pt]{amsart}
\usepackage{amsmath,amssymb,amsthm,amsfonts,setspace,hyperref,dsfont,yhmath}

\newcommand{\NN}{\mathbb{N}}
\newcommand{\RR}{\mathbb{R}}

\newcommand{\CC}{\mathbb{C}}

\newcommand{\ZZ}{\mathbb{Z}}

\newcommand{\norm}[1]{\lVert#1\rVert}
\newcommand{\abs}[1]{\lvert#1\rvert}

\newtheorem{theorem}{Theorem}[section]
\newtheorem{corollary}[theorem]{Corollary}

\newtheorem{lemma}[theorem]{Lemma}
\newtheorem{proposition}[theorem]{Proposition}

\newcommand{\comment}[1]{}
\theoremstyle{definition}
\newtheorem{definition}[theorem]{Definition}
\newtheorem{remark}[theorem]{Remark}

\numberwithin{equation}{section}

\makeatletter
\renewcommand*\env@matrix[1][*\c@MaxMatrixCols c]{%
  \hskip -\arraycolsep
  \let\@ifnextchar\new@ifnextchar
  \array{#1}}
\makeatother
\begin{document}
\title[Twisted geodesic flow]
{The twisted cohomological equation over the geodesic flow}

\author{ Zhenqi Jenny Wang$^1$ }
\thanks{ $^1$ Based on research supported by NSF grant   DMS-1700837}
\address{Department of Mathematics\\
        Michigan State University\\
        East Lansing, MI 48824,USA}
\email{wangzq@math.msu.edu}

\begin{abstract}
We study the twisted cohomoligical equation over the geodesic flow on $SL(2,\RR)/\Gamma$. We characterize the obstructions to solving the twisted
cohomological equation, construct smooth solution and obtain the tame
Sobolev estimates for the solution, i.e, there is finite loss of regularity (with
respect to Sobolev norms) between the twisted coboundary and the solution. We also give a tame splittings for non-homogeneous cohomological equations. The result can be viewed as a
first step toward the application of KAM method in obtaining differential rigidity for partially hyperbolic actions in products of rank-one groups in future works.
\end{abstract}

\maketitle
\section{Introduction}

\subsection{Motivation and results} The cohomological equations of
the hococycle flow and the geodesic flow of the
homogeneous spaces of $SL(2,\RR)$ have been well understood, see \cite{Forni} and \cite{Mieczkowski}. In this paper, we extend the study to the twisted cohomological equation of the geodesic flow.

In fact, the study of the twisted  cohomological equation provides a tool for obtaining local differentiable rigidity of algebraic actions by KAM type iteration
scheme. The KAM scheme was used by Damjanovic and Katok to prove local rigidity for genuinely higher-rank partially hyperbolic actions on torus in \cite{Damjanovic4}.
Later, an improved version of the scheme was applied on homogeneous space of $SL(2,\RR)\times SL(2,\RR)$ to obtain weak local rigidity for certain
parabolic algebraic actions \cite{DK-parabolic}. To carry out the scheme, people need to solve the
linearized equation:
\begin{align*}
 \text{Ad}(\alpha)\Lambda-\Lambda\circ\alpha=0
\end{align*}
over the algebraic action $\alpha$, where $\Lambda$ is valued on the tangent space of the homogeneous space. The equation decomposes into the twisted cohomological equations of the form
\begin{align}\label{for:7}
 \mu\Lambda_i-\Lambda_i\circ\alpha=0
\end{align}
on the $\mu$-eigenspace of $\text{Ad}(\alpha)$. Hence a complete and detailed description of twisted cohomological obstructions for the action $\alpha$ is necessary for the scheme.

In this paper, we give a complete solution to the twisted cohomological equation over the geodesic flow. We construct of the solution to the twisted coboundary equation, classify the
the obstructions and obtain tame estimates of the solution. The results in the present paper will be used to prove local differentiable rigidity of the left translations of the two-dimensional
subgroup
\begin{align*}
\begin{pmatrix}e^t & 0 \\
0 & e^{-t} \end{pmatrix}\times \begin{pmatrix}e^s & 0 \\
0 & e^{-s} \end{pmatrix}\times\begin{pmatrix}e^s & 0 \\
0 & e^{-s}  \end{pmatrix}
\end{align*}
where $s,\,t\in\RR$ on $SL(2,\RR)^3/\Gamma$, see \cite{Zhenqi4}.

\subsection{History and method} Results concerning the cohomology of horocycle flow are due to Flaminio and Forni in \cite{Forni}. They used Fourier analysis in each irreducible
unitary representations of $PSL(2,\RR)$ to obtain  Sobolev estimates of the
cohomological equation. These estimates satisfy a uniform upperbound condition, across the class of irreducible representations. Global estimates were then formed by
glueing estimates together from each irreducible component. This scheme was further used
in \cite{T1} to study the cohomological equation of
the classical (discrete) horocycle map, % on $PSL(2,\RR)/\Gamma$,
and it was also used in \cite{Mieczkowski} to study the cohomological equation of
the classical geodesic flow.

In this paper, we follow the same general scheme as in \cite{Forni} to study the twisted cohomological equation over the geodesic flow. In earlier papers, the obstructions to solving the
equation can be constructed explicitly, which provides distributional solutions by Green's function. For the twisted equations, the obstructions are much more complex, which results in explicit construction is mostly likely impossible. This does seem
to require some new techniques for handling it; the same is true in an attempt
at obtaining Sobolev estimates of the solution.

\section{Statement of results}

\subsection{Irreducible representations of $SL(2,\RR)$}\label{sec:1} We choose as generators for $\mathfrak{sl}(2,\RR)$ the elements
\begin{align}\label{for:4}
X=\begin{pmatrix}
  1 & 0 \\
  0 & -1
\end{pmatrix},\quad U=\begin{pmatrix}
  0 & 1 \\
  0 & 0
\end{pmatrix},\quad V=\begin{pmatrix}
  0 & 0 \\
  1 & 0
\end{pmatrix}.
\end{align}
The \emph{Casimir} operator is then given by
\begin{align*}
\Box:= -X^2-2(UV+VU),
\end{align*}
which generates the center of the enveloping algebra of $\mathfrak{sl}(2,\RR)$. The Casimir operator $\Box$
acts as a constant $u\in\RR$ on each irreducible unitary representation space  and its value classifies them into three classes except the trivial representation.
For \emph{Casimir parameter} $\mu$ of $SL(2,\RR)$, let $\nu=\sqrt{1-\mu}$ be a representation parameter. We denote by $(\pi_\nu,\mathcal{H}_\nu)$ or $(\pi_\mu,\mathcal{H}_\mu)$ the following models for the
\begin{enumerate}
  \item principal series ($\nu\in i\RR$);

  \smallskip
  \item complementary series ($\nu\in (-1,1)\backslash 0$);

  \smallskip
  \item the mock discrete series or the principal series ($\nu=0$);

  \smallskip
  \item discrete series representation spaces ($\nu\in \ZZ\backslash 0$).

\end{enumerate}
For the principal series, we also use the notation $(\pi^+_\nu,\mathcal{H}^+_\nu)$ for the  spherical model and $(\pi^-_\nu,\mathcal{H}^-_\nu)$ for the non-spherical model. For the discrete series we also use $(\pi^+_\nu,\mathcal{H}^+_\nu)$ to denote the  upper half-plane model and $(\pi^-_\nu,\mathcal{H}^-_\nu)$ to denote the lower half-plane model.

Any unitary representation $(\pi,\mathcal{H})$ of $SL(2,\RR)$ is decomposed into a direct integral (see \cite{Forni} and \cite{Mautner})
\begin{align}\label{for:1}
\mathcal{H}=\int_{\oplus}\ell(\mu)\mathcal{H}_\mu dS(\mu)
\end{align}
with respect to a positive Stieltjes measure $dS(\mu)$ over the spectrum $\sigma(\Box)$. The
Casimir operator acts as the constant $\mu\in \sigma(\Box)$ on every Hilbert space $\mathcal{H}_\mu$. Here $\ell(\mu)$ is the (at most countable) multiplicity of the
irreducible representation of $SL(2,\RR)$ appearing in $\pi$. We say that \emph{$\pi$ has a spectral gap (of $u_0$)} if $u_0>0$ and $S((0,u_0])=0$ and $\pi$ contains no non-trivial $SL(2,\RR)$-fixed vectors.

In this paper, we only consider unitary representations of $SL(2,\RR)$ with a spectral gap. That is, for complementary series, we assume there is $0<\mathfrak{u}_0<1$ such that $\nu\in (-\mathfrak{u}_0,\mathfrak{u}_0)\backslash 0$. For the proofs involving the discrete series, we only consider the holomorphic case ($\nu\geq 1$) because there is a complex antilinear
isomorphism between two series of the same Casimir parameter, but we list corresponding results for the anti-holomorphic case ($\nu\leq -1$).

\subsection{Statement of the results} For any unitary representation $(\pi,\mathcal{H})$ of $SL(2,\RR)$ let
\begin{align*}
   \mathcal{H}_{X-m}^{-k}&=\{\mathcal{D}\in \mathcal{H}^{-k}: (X-m)\mathcal{D}=0 \}.
\end{align*}
The next result characterizes the obstructions to solving the twisted cohomological equation and obtains Sobolev estimates for the solution.
\begin{theorem}\label{th:1}
Suppose $(\pi,\mathcal{H})$ is a unitary representation of $SL(2,\RR)$ with a spectral gap. For the twisted cohomological equation $(X+m)f=g$, $m\in\RR\setminus 0$ we have
\begin{enumerate}
\item \label{for:17}if $g\in \mathcal{H}$,  the equation has a unique solution $f\in \mathcal{H}$ with
\begin{align*}
 \norm{f}\leq |m|^{-1}\norm{g};
\end{align*}

  \item\label{for:22} if $g\in \mathcal{H}^s$ with $s\geq\frac{\abs{m}}{2}+8$, and $\mathcal{D}(g)=0$ for any $(X-m)$-invariant distribution $\mathcal{D}$, then the equation has a solution $f\in \mathcal{H}^{s-\frac{\abs{m}}{2}-3}$ with estimates
\begin{align*}
\norm{f}_t\leq C_m \norm{g}_{t+\frac{\abs{m}}{2}+3},\qquad t\leq s-\text{\small$\frac{\abs{m}}{2}$}-3;
\end{align*}
  \item\label{for:23} if $g\in \mathcal{H}^s$ with $s\geq\frac{\abs{m}}{2}+8$, and  the equation $(X+m)f=g$ has a solution $f\in \mathcal{H}^{\frac{\abs{m}+10}{2}}_\nu$, then $f\in \mathcal{H}^{s-\frac{\abs{m}}{2}-3}$ with estimates
\begin{align*}
\norm{f}_t\leq C_m \norm{g}_{t+\frac{\abs{m}}{2}+3},\qquad t\leq s-\text{\small$\frac{\abs{m}}{2}$}-3.
\end{align*}
\end{enumerate}
\end{theorem}
\begin{remark}
In \cite{Ramirez}, Ramirez shows that for the regular representation of $SL(2,\RR)$ the existence of an $L^2$-solution of the cohomological equation $Xf=g$ grantees the existence of an smooth solution if $g$ is smooth.
This is quite different from the twisted case, since the above theorem shows that for the twisted cohomological equation, an $L^2$-solution always exists.
\end{remark}
The next two theorems make a detailed study for the twisted equation in each non-trivial irreducible component of $SL(2,\RR)$. Also, tame splittings are provided for non-homogeneous equations.

Let $I_\nu=2\ZZ$ or $2\ZZ+1$ if $\mu$ parametrizes the principal series, or let $I_\nu=2\ZZ$ if $\mu$ parametrizes the complimentary series or let $I_\nu=[n,\infty]\subset\ZZ^+$ or $I_\nu=[-\infty,n]\subset\ZZ^-$
if $\mu$ parametrizes the holomorphic discrete series.

For any $f=\sum_{k\in I_\nu}f_ku_k\in \mathcal{H}_\nu$ (see Section \ref{sec:4}) and $n\in I_v\backslash 0$, set
\begin{align}\label{for:107}
&f|_n=\left\{\begin{aligned} &\sum_{k\in I_\nu,\,k\geq n}f_ku_k,&\quad &n>0\\
&\sum_{k\in I_\nu,\,k\leq n}f_ku_k,&\quad& n<0.
\end{aligned}
 \right.
\end{align}
Set $\mathcal{S}^+_\nu=\{0,\,2\}$ (resp. $\mathcal{S}^-_\nu=\{-1,\,1\}$) if $\nu\in i\RR$; $\mathcal{S}^+_\nu=\{0,\,2\}$ (resp. $\mathcal{S}^-_\nu=\emptyset$) if $\nu\in (-\mathfrak{u}_0,\mathfrak{u}_0)\backslash 0$; and $\mathcal{S}^+_\nu=\{\nu+1\}$ (resp. $\mathcal{S}^-_\nu=\{\nu-1\}$) if $\nu\in\ZZ^+\cup 0$ (resp. $\nu\in\ZZ^-\cup0$).

\begin{theorem}\label{th:3} Suppose $0\neq m\in\RR$. In any non-trivial irreducible representation $(\pi^\delta_\nu,\mathcal{H}^\delta_\nu)$, $\delta=\pm$, $\nu\in i\RR\cup\ZZ\cup(-u_0,u_0)$ of $SL(2,\RR)$, there exists $\mathcal{D}_{\nu,n}^{\delta,m}\in (\mathcal{H}^\delta_\nu)_{X-m}^{\text{\tiny$-\frac{|m|+8}{2}$}}$ where $\delta=\pm$, $n\in \mathcal{S}^\delta_\nu$ such that for any $g\in (\mathcal{H}^\delta_\nu)^s$, $s\geq\frac{\abs{m}}{2}+8$, $\delta=\pm$ we have
\begin{enumerate}
  \item\label{for:111} the equation
      \begin{align*}
        (X+m)f=g+\sum_{n\in \mathcal{S}^\delta_\nu}\mathcal{D}_{\nu,n}^{\delta,m}(g)u_n
      \end{align*}
      has a solution $f\in (\mathcal{H}^\delta_\nu)^{s-\frac{\abs{m}}{2}-3}$ with estimates
\begin{align*}
\norm{f}_t\leq C_m \norm{g}_{t+\frac{\abs{m}}{2}+3},\qquad t\leq s-\text{\small$\frac{\abs{m}}{2}$}-3.
\end{align*}
Furthermore, if we write $f=\sum_{n\in I_\nu}f_nu_n\in \mathcal{H}_\nu$ we have
\begin{align*}
 &\norm{f_nu_n}_t\leq\left\{\begin{aligned}&C_{m}\norm{\Theta^{\frac{\abs{m}+3}{2}}(g|_{n-2})}_{t+\frac{1}{2}}; &\quad &\text{if }n<0\\
&C_{m}\norm{\Theta^{\frac{\abs{m}+3}{2}}(g|_{n+2})}_{t+\frac{1}{2}},&\quad& \text{if }n> 0,
\end{aligned}
 \right.
\end{align*}
where $\Theta=U-V$, for any $t\leq s-\text{\small$\frac{\abs{m}}{2}$}-3$;
\medskip
  \item \label{for:112} if the equation $(X+m)f=g$ has a solution $f\in \mathcal{H}^{\frac{\abs{m}+10}{2}}_\nu$ then for $n\in \mathcal{S}^\delta_\nu$, $\text{\small$\mathcal{D}_{\nu,n}^{\delta,m}(g)$}=0$;

  \medskip
  \item\label{for:113} for any $g\in (\mathcal{H}^\pm_\nu)^{s}$, $s\geq\frac{\abs{m}+8}{2}$ and $n\in \mathcal{S}^\pm_\nu$ we have
  \begin{align*}
      \norm{\mathcal{D}_{\nu,n}^{\pm,m}(g)u_n}_t\leq C_m\norm{g}_{\frac{\abs{m}+8}{2}+t},
      \end{align*}
 for any $0\leq t\leq s-\frac{\abs{m}+8}{2}$.
\end{enumerate}

\end{theorem}

The case of $m=0$ is proved in \cite{Mieczkowski}, see Theorem 4.1 and Theorem 4.2.

\begin{theorem}\label{th:7}
Suppose $g\in (\mathcal{H}_\nu)^s$, $s\geq\frac{\abs{m}}{2}+8$, $0\neq m\in\RR$. For any $n\in I_\nu\backslash 0$ with $\abs{n}\geq \abs{\text{Re}(\nu)}+2$ there exists
\begin{align*}
 &\tilde{g}=\left\{\begin{aligned}&a_1u_n+a_2u_{n-2}; &\quad &\text{if }n>0\\
&b_1u_n+b_2u_{n+2},&\quad& \text{if }n<0,
\end{aligned}
 \right.
 \end{align*}
where $a_1,a_2,b_1,b_2\in\CC$ such that the equation
\begin{align*}
 (X+m)f=g|_n-\tilde{g}
\end{align*}
has a solution $f\in (\mathcal{H}_\nu)^{s-\frac{\abs{m}}{2}-3}$ such that $f=f|_n$ with estimates
 \begin{align*}
 \norm{f}_t\leq\norm{g|_n}_{t+\frac{\abs{m}}{2}+3},\qquad t\leq s-\text{\small$\frac{\abs{m}}{2}$}-3.
 \end{align*}
 Moreover, we have
 \begin{align*}
\norm{\tilde{g}}_t\leq C_m\norm{g|_n}_{t+\frac{\abs{m}}{2}+4}
\end{align*}
if $t\leq s-\text{\small$\frac{\abs{m}}{2}$}-4$.
\end{theorem}

\section{Preliminaries on representation theory of $SL(2,\RR)$}

\subsection{Sobolev spaces}\label{sec:17}  As in Flaminio-Forni \cite{Forni}, the
Laplacian gives unitary representation spaces a natural Sobolev structure. Let $\pi$ be a unitary representation of $SL(2,\RR)$  on a
Hilbert space $\mathcal{H}$. The Sobolev space of order $s>0$ is the Hilbert space $\mathcal{H}^s\subset \mathcal{H}$ that is the maximal domain determined by the inner product
\begin{align*}
\langle v_1,\,v_2\rangle_s=\langle (I+\Delta^s)v_1,\,v_2\rangle
\end{align*}
for any $v_1,v_2\in \mathcal{H}$.

The subspace $\mathcal{H}^\infty$
coincides with the intersection of the spaces $\mathcal{H}^s$ for all $s\geq 0$. $\mathcal{H}^{-s}$, defined as the Hilbert space duals of
the spaces $\mathcal{H}^{s}$, are subspaces of the space $\mathcal{E}(\mathcal{H})$ of distributions, defined as the
dual space of $\mathcal{H}^\infty$.

In addition to the decomposition \eqref{for:1}, all the operators in the enveloping algebra are decomposable with respect to the direct integral decomposition \eqref{for:1}. Hence there exists for all $s\in\RR$ an induced direct
decomposition of the Sobolev spaces:
\begin{align}\label{for:51}
\mathcal{H}^s=\int_{\oplus}\ell(\mu)\mathcal{H}^s_\mu dS(\mu)
\end{align}
with respect to the measure $dS(\mu)$ (we refer to
\cite[Chapter 2.3]{Zimmer} or \cite{Margulis} for more detailed account for the direct integral theory).

The existence of the direct integral decompositions
\eqref{for:1}, \eqref{for:51} allows us to reduce our analysis of the
cohomological equation to irreducible unitary representations. This point of view is
essential for our purposes.

\subsection{Sobolev norms}\label{sec:4} There exists an orthogonal basis $\{u_k\}$ in $\mathcal{H}_\nu$, basis of eigenvectors of the operator $\Theta=U-V$ and hence of the Laplacian operator $\Delta=\Box-2\Theta^2$, satisfying:
\begin{align}\label{for:110}
 \Theta u_k=iku_k,\qquad \Delta u_k=(\mu+2k^2)u_k;
\end{align}
and the norms of the $u_k$ are given recursively by
\begin{align}\label{for:18}
&\norm{u_k}^2=\left\{\begin{aligned} &\norm{u_{k-2}}^2,&\quad &\nu\in i\RR\\
&\frac{\abs{k}-1-\nu}{\abs{k}-1+\nu}\norm{u_{k-2}}^2,&\quad& \nu\in\RR,
\end{aligned}
 \right.
\end{align}
see \cite{Forni}. By defining $\Pi_{\nu,k}=\Pi_{j=i_\nu}^k\frac{\abs{k}-1-\nu}{\abs{k}-1+\nu}$, for any integer $k\geq i_\nu=1+\text{Re}(\nu)$ (Empty products are set equal to $1$; hence, if $k=i_\nu$, then $\Pi_{\nu,k}=1$ in all cases)  we get that
\begin{align*}
  \norm{u_k}^2=\abs{\Pi_{\nu,\abs{k}}}.
\end{align*}
From Section \ref{sec:17} the Sobolev norms of the vectors of the orthogonal basis $\{u_k\}$ are given by the identities
\begin{align*}
 \norm{u_k}^2_s=\langle (I+\Delta^s)u_k,\,u_k\rangle=(1+\mu+2k^2)\norm{u_k}^2.
\end{align*}
Then the Sobolev norm of a vector $f=\sum_kf_ku_k\in \mathcal{H}^s_\nu$ is:
\begin{align}\label{for:9}
 \norm{f}_{s}=\big( \sum_k(1+\mu+2k^2)^s\abs{\Pi_{\nu,k}}\abs{f_k}^2\big)^{\frac{1}{2}}.
\end{align}
\begin{lemma}\label{le:10}(Lemma 2.1 of \cite{Forni}) If $\nu\in i\RR$, the for all $k\geq0$,
\begin{align*}
\abs{\Pi_{\nu,k}}=1.
\end{align*}
There exists $C>0$ such that, if $\nu\in(-1,1)\backslash 0$ for all $k>0$, we have
\begin{align*}
  C^{-1}\big( \frac{1-\nu}{1+\nu}\big)(1+k)^{-\nu}\leq\Pi_{\nu,k}\leq C\big( \frac{1-\nu}{1+\nu}\big)(1+k)^{-\nu};
\end{align*}
if $\nu\in \ZZ^+$ for all $k\geq \ell\geq i_\nu=\nu+1=n$, we have
\begin{align*}
  C^{-1}\big(\frac{k-n+2}{\ell-n+2}\big)^{-\nu}\leq\frac{\Pi_{\nu,k}}{\Pi_{\nu,\ell}}\leq C\big( \frac{k-n+2\nu}{\ell-n+2\nu}\big)^{-\nu}.
\end{align*}
\end{lemma}
By the above lemma, $\norm{u_k}_s^2\approx (1+\abs{k})^{2s-\text{Re}(\nu)}$. So it follows that,
\begin{align}\label{for:16}
 \norm{f}_s\approx \big( \sum_k(1+\abs{k})^{2s-\text{Re}(\nu)}\abs{f_k}^2\big)^{\frac{1}{2}}.
\end{align}

\section{The basic solutions of the twisted equation }
In this section we study the twisted cohomological equation
\begin{align}\label{for:105}
 (X+m)f=g
\end{align}
$m\in \RR$ of the classical geodesic flow defined by the $\mathfrak{sl}(2,\RR)$-matrix $X=\begin{pmatrix}
  1 & 0 \\
  0 & -1
\end{pmatrix}$ in each $\mathcal{H}_\nu$. The action
of $X$ on the basis element $\{u_k\}$ is given by:
\begin{lemma}\label{le:7} (Lemma 3.4 of \cite{Forni}) We have
 \begin{align*}
  (X+m)u_k=\text{\small$\frac{k+1+\nu}{2}$}u_{k+2}+mu_k-\text{\small$\frac{k-1-\nu}{2}$}u_{k-2},\qquad \forall\,k\in I_\nu.
 \end{align*}
For $n\in\NN$ $\nu=n-1$ and $k=n$, the above equation  must be read as $(X+m)u_n=mu_n+nu_{n+2}$.
\end{lemma}
Let $f=\sum_k f_ku_k$ and $g=\sum_kg_ku_k$ be the Fourier expansions of the distributions $f$, $g$ with respect to the adapted basis of $\mathcal{H}_\nu$. So
the twisted equation \eqref{for:105} becomes
\begin{align}\label{for:2}
  g_k=-\text{\small$\frac{k+1-\nu}{2}$} f_{k+2}+mf_k+\text{\small$\frac{k-1+\nu}{2}$}f_{k-2}
\end{align}
for all $k\in I_\nu$; for $\nu=n-1$ (discrete series) and $k=n$ equation \eqref{for:2} should be read as $g_n=mf_n-f_{n+2}$.

\begin{definition}\label{de:1}
We say that a vector $f=\sum_kf_ku_k$ in $\mathcal{H}_\nu$ is $\Theta$-finite if where exists $\ell\in\NN$ such that $f_k=0$ if $\abs{k}\geq\ell$. It is clear that if $f$ is
$\Theta$-finite then $f\in \mathcal{H}^\infty_\nu$.
\end{definition}
From \eqref{for:2} we have a simple \textbf{observation}: if $(X+m)f=g$, where $g=\sum_{k\in \mathcal{S}_\nu}g_ku_k$ and $f$ is $\Theta$-finite, then $f=g=0$.
\subsection{Uniqueness of the solution}
For the twisted equation \eqref{for:105}, if $m=0$,  in each non-trivial irreducible component of $SL(2,\RR)$, the uniqueness of the solution is guaranteed by the ergodicity of the geodesic flow \cite{howe-moore}. For the case of $m\neq0$, the next results show that it it unique in any representations.
\begin{lemma}\label{le:4}
Suppose $(\pi,\mathcal{H})$ is a unitary representation of $SL(2,\RR)$. Then for any $g\in \mathcal{H}$ and any
$m\in\RR\backslash 0$, the twisted equation
\begin{align*}
 (X+m)f=g
\end{align*}
has a unique solution $f\in \mathcal{H}$ with
\begin{align*}
  \norm{f}\leq \abs{m}^{-1}\norm{g}.
\end{align*}

\end{lemma}
\begin{proof}
In $SL(2,\RR)$ let $X_t$ denote the  classical geodesic flow $\Big\{\begin{pmatrix}
  e^t & 0 \\
  0 & e^{-t}
\end{pmatrix}\Big\}_{t\in\RR}$. Since $X_t$ is isomorphic to $\RR$ we have a direct integral decomposition
\begin{align*}
    \pi\mid_{X_t}=\int_{\widehat{\RR}}\chi(t)du(\chi)
\end{align*}
where $u$ is a regular Borel measure and
\begin{align*}
    v=\int_{\widehat{\RR}}v_{\chi}du(\chi),\qquad \forall\, v\in \mathcal{H}.
\end{align*}
Set
\begin{align*}
  f_{\chi}=(m+\chi'(0))^{-1}g_{\chi},\qquad \chi\in \widehat{\RR}.
\end{align*}
We see that $f=\int_{\widehat{\RR}}(m+\chi'(0))^{-1}g_{\chi}du(\chi)$ is a formal solution of the equation $(X+m)f=g$.

Next, we will show that $f\in \mathcal{H}$.  Since $\chi'(0)\in i\RR$ and $|m+\chi'(0)|\geq \abs{m}$ for any $\chi\in \widehat{\RR}$, we have
\begin{align*}
 \norm{f}^2=\int_{\widehat{\RR}}\abs{m+\chi'(0)}^{-2}\norm{g_{\chi}}^2du(\chi)\leq \abs{m}^{-2}\int_{\widehat{\RR}}\norm{g_{\chi}}^2du(\chi)=\abs{m}^{-2}\norm{g}^2.
\end{align*}
This shows that $f\in \mathcal{H}$.

On the other hand, if $(X+m)f=0$ with $f\in \mathcal{H}$, then we have
\begin{align*}
  (m+\chi'(0))f_{\chi}=0
\end{align*}
for almost every $\chi\in \widehat{\RR}$ with respect to $u$. This implies that $f_{\chi}=0$ for almost every $\chi\in \widehat{\RR}$. Then we have $f=0$. Hence we showed the uniqueness of the solution of the twisted equation. This completes the proof.

\end{proof}

\subsection{Explicit construction of basic solutions}\label{sec:2} For any $m\in \RR$ and $n\in I_\nu\backslash \mathcal{S}_\nu$, we want to find the vector
\begin{align}\label{for:8}
 &\mathfrak{U}_n=\left\{\begin{aligned}&u_n+d_0u_0+d_2u_2; &\quad &\text{ if }u_n\in \mathcal{H}^+_\nu,\,\nu\in i\RR\cup(-1,1)\\
&u_n+d_{-1}u_{-1}+d_1u_1,&\quad& \text{ if }u_n\in \mathcal{H}^-_\nu,\,\nu\in i\RR\cup(-1,1)\\
&u_n+d_{\nu+1}u_{\nu+1},&\quad& \text{ if }u_n\in \mathcal{H}^+_\nu,\,\nu\in \ZZ^+\cup 0,
\end{aligned}
 \right.
\end{align}
such that there is a $\Theta$-finite $f_{\{n\}}$ such that
\begin{align*}
 (X+m)f_{\{n\}}=\mathfrak{U}_n.
\end{align*}
We write
\begin{align}\label{for:44}
 &f_{\{n\}}=\left\{\begin{aligned}&\sum_{\substack{2\leq 2k\leq -n\\ n+2k\in I_\nu}}b_{n,n+2k}u_{n+2k}; &\quad &\text{if }n<0\\
&\sum_{\substack{2\leq 2k<n\\ n-2k\in I_\nu}}b_{n,n-2k}u_{n-2k},&\quad& \text{if }n> 0.
\end{aligned}
 \right.
\end{align}
By \eqref{for:2}, if $n\geq 1$ it is easy to see that
\begin{align}\label{for:12}
 b_{n,n-2}=\text{\small $\frac{2}{n-1+\nu}$},\quad b_{n,n-4}=-\text{\small $\frac{4m}{(n-1+\nu)(n-3+\nu)}$};
\end{align}
and for any $k\geq 1$ with $n-2k\geq 0$, we can obtain the sequence $(b_{n,n-2k})$ using a recursive rule  along with the two initial elements $b_{n,n-2}$ and $b_{n,n-4}:$
\begin{align*}
 b_{n,n-2k}=-\text{\small $\frac{2m}{n-2k+1+\nu}$}b_{n,n-2k+2}+\text{\small $\frac{n-2k+3-\nu}{n-2k+1+\nu}$}b_{n,n-2k+4}.
\end{align*}
If $n\leq -1$ we get that
\begin{align}\label{for:19}
  b_{n,n+2k}=b_{-n,-n-2k},\qquad\text{if }n<0,\,k\geq1, \text{ and }n+2k\leq 0.
\end{align}
Set
\begin{align}\label{for:43}
\mathfrak{U}_n= (X+m)f_{\{n\}}.
\end{align}
It is clear that $\mathfrak{U}_n$ has the same form as in \eqref{for:8}.

These $f_{\{n\}}$, $n\in\ZZ$ are called \emph{basic solutions} and will be used to construct explicit solution of the twisted equation \eqref{for:105} in next Section \ref{sec:5}.
We now make a slight digression to obtain upperbounds of
\begin{align*}
 \abs{b_{n,n-2k}}\text{\small$\frac{\norm{u_{n-2k}}}{\norm{u_n}}$},
\end{align*}
which will be used to estimate the Sobolev orders of the basic solutions, see Proposition \ref{le:11}.
\begin{proposition}\label{po:1}
For $\nu\in i\RR\cup\ZZ^+\cup(-u_0,u_0)$ and $n-2k\in I_\nu$ we have
\begin{align*}
 \abs{b_{n,n-2k}}\text{\small$\frac{\norm{u_{n-2k}}}{\norm{u_n}}$}\leq C_m\abs{n}^{\frac{\abs{m}+2}{2}}.
\end{align*}
\end{proposition}
From \eqref{for:19} it suffices to consider the case of  $n\geq0$ and $n-2k\geq0$. The remaining part of this section will be dedicated to the proof of this proposition.
\subsection{Upperbounds for $\nu\in i\RR$} We set
\begin{align}\label{for:47}
  c_{n-2k}&=\text{\small$\frac{2}{\abs{n-2k+1+\nu}}\Pi_{\ell=1}^{k}\big(1+\frac{\abs{m}}{\abs{n-2\ell+1+\nu}}\big)$}.
\end{align}
\begin{lemma}\label{le:5}
Suppose $\nu\in i\RR$ and $n-2k\geq \max\{6, m^2\}$, then
\begin{align*}
 \abs{b_{n,n-2k}}\leq c_{n-2k}.
\end{align*}

\end{lemma}
\begin{proof}
We prove by induction. From \eqref{for:12} we see that that
\begin{align*}
 \abs{b_{n,n-2k}}\leq c_{n-2k},\qquad k=1,\,2.
\end{align*}
Suppose
\begin{align}\label{for:11}
 \abs{b_{n,n-2k}}\leq c_{n-2k}, \qquad 2\leq k\leq r,
\end{align}
where $n-2r\geq \max\{6, m^2\}$.

If $n-2(r+1)\geq \max\{6, m^2\}$ then
\begin{align*}
 \abs{b_{n,n-2r-2}}&\leq \big|\text{\small $\frac{n-2r+1-\nu}{n-2r-1+\nu}$}b_{n,n-2r+2}\big|+\big|\text{\small $\frac{2m}{n-2r-1+\nu}$}b_{n,n-2r}\big|\\
 &\overset{\text{(1)}}{\leq} \big|\text{\small $\frac{n-2r+1-\nu}{n-2r-1+\nu}$}\big|c_{n-2r+2}+\big|\text{\small $\frac{2m}{n-2r-1+\nu}$}\big| c_{n-2r} \\
 &= \Big|\text{\small $\frac{n-2r+1-\nu}{n-2r-1+\nu}$}\Big|\cdot\text{\small $\frac{2}{\abs{n-2r+3+\nu}}$}\Pi_{\ell=1}^{r-1}\big(1+\text{\small $\frac{\abs{m}}{\abs{n-2\ell+1+\nu}}$}\big)\\
 &+\Big|\text{\small $\frac{2m}{n-2r-1+\nu}$}\Big|\cdot\text{\small $\frac{2}{\abs{n-2r+1+\nu}}$}\Pi_{\ell=1}^{r}\big(1+\text{\small $\frac{\abs{m}}{\abs{n-2\ell+1+\nu}}$}\big)\\
&= \Big|\text{\small $\frac{2}{n-2r-1+\nu}$}\Big|\Pi_{\ell=1}^{r-1}\big(1+\text{\small $\frac{\abs{m}}{\abs{n-2\ell+1+\nu}}$}\big)\\
 &\cdot\big(\text{\small $\frac{\abs{n-2r+1-\nu}}{\abs{n-2r+3+\nu}}$}+\text{\small $\frac{2\abs{m}}{\abs{n-2r+1+\nu}}$}\big(1+\text{\small $\frac{\abs{m}}{\abs{n-2r+1+\nu}}$}\big)\big).
\end{align*}
In $(1)$ we use the assumption \eqref{for:11}.

Hence to show that $\abs{b_{n,n-2r-2}}\leq c_{n-2r-2}$, it suffices to prove that
\begin{align}\label{for:34}
 &\Big|\text{\small $\frac{n-2r+1-\nu}{n-2r+3+\nu}$}\Big|+\text{\small $\frac{2\abs{m}}{\abs{n-2r+1+\nu}}$}\big(1+\text{\small $\frac{\abs{m}}{\abs{n-2r+1+\nu}}$}\big)\notag\\
 &\leq \Pi_{\ell=r}^{r+1}\big(1+\text{\small $\frac{\abs{m}}{\abs{n-2\ell+1+\nu}}$}\big).
\end{align}
In fact, for any $x\geq \max\{5, m^2-1\}$ and $b\in\RR$ we have
\begin{align*}
 \text{\small $\frac{\sqrt{x^2+b^2}}{\sqrt{(x+2)^2+b^2}}$}&=\text{\small $\sqrt{1-\frac{4x+4}{(x+2)^2+b^2}}$}\overset{\text{(1)}}{\leq} 1-\text{\small $\frac{2x+2}{(x+2)^2+b^2}$}\\
 &\overset{\text{(2)}}{\leq} 1-\text{\small $\frac{m^2}{x^2+b^2}$}=\big(1-\text{$\small \frac{\abs{m}}{\sqrt{x^2+b^2}}$}\big)\big(1+\text{\small $\frac{\abs{m}}{\sqrt{x^2+b^2}}$}\big)\\
 &\leq \big(1+\text{\small $\frac{\abs{m}}{\sqrt{x^2+b^2}}$}\big)\big(1+\text{\small $\frac{\abs{m}}{\sqrt{(x-2)^2+b^2}}$}-\text{\small $\frac{2\abs{m}}{\sqrt{x^2+b^2}}$}\big).
\end{align*}
In $(1)$ we use the inequality $\sqrt{1-y}\leq 1-\frac{1}{2}y$ if $0\leq y\leq 1$; in $(2)$ we use the assumption that $x\geq \max\{5, m^2-1\}$.

Then \eqref{for:34} follow immediately from the above inequality by letting $b=\abs{\nu}$ and $x=n-2r+1$. Hence we finish the proof.
\end{proof}
\begin{corollary}\label{le:8}
Suppose $\nu\in i\RR$. If $n\geq0$ and $n-2k\geq0$ then
\begin{align*}
 \abs{b_{n,n-2k}}\text{\small$\frac{\norm{u_{n-2k}}}{\norm{u_n}}$}\leq C_m\abs{n}^{\frac{\abs{m}}{2}}.
\end{align*}
\end{corollary}
\begin{proof}
We note that if $n\geq0$ and $n-2k\geq2$ then
\begin{align*}
 &\log \Big(\Pi_{\ell=1}^{k}\big(1+\text{\small $\frac{\abs{m}}{\abs{n-2\ell+1+\nu}}$}\big)\Big)=\sum_{\ell=1}^{k}\log\big(1+\text{\small $\frac{\abs{m}}{\abs{n-2\ell+1+\nu}}$}\big)\\
 &\overset{\text{(1)}}{\leq} \int_{1}^{k+1}\text{\small $\frac{\abs{m}}{n-2\ell+1}$}d\ell=\frac{\abs{m}}{2}\log\big(\text{\small $\frac{n-1}{n-2k-1}$}\big).
\end{align*}
In $(1)$ we use the inequality $\log(1+x)\leq x$, for any $x\geq0$ and integral inequalities.

The above inequality shows that if $n-2k\geq 2$ then
\begin{align*}
\abs{c_{n-2k}}\leq \text{\small $\frac{2}{n-2k+1}$}(\text{\small $\frac{n-1}{n-2k-1}$})^{\frac{\abs{m}}{2}}.
\end{align*}
By Lemma \ref{le:10} we have $\norm{u_{n-2k}}=\norm{u_n}$. Then the above estimates and Lemma \ref{le:5} imply the conclusion.
\end{proof}
\subsection{Upperbounds for $\nu\in (-1,1)\backslash 0$} Let
\begin{align*}
c_{n-2k}=\text{\small $\frac{4}{n-2k+1+\nu}$}\Pi_{\ell=1}^{k}(1+\text{\small $\frac{\abs{m}+2-\nu}{n-2\ell-1+\nu}$}).
\end{align*}
\begin{lemma}\label{le:6}
Suppose $\nu\in (-1,1)\backslash 0$ and $n-2k-\nu>\max\{7,\,2\abs{m}(\abs{m}+2)\}$. Then $\abs{b_{n,n-2k}}\leq c_{n-2k}$.
\end{lemma}
\begin{proof}
We prove by induction. From \eqref{for:12} it is easy to check that
\begin{align*}
 \abs{b_{n,n-2k}}\leq c_{n-2k},\qquad k=1,\,2.
\end{align*}
if $n-2k-\nu>\max\{7,\,2\abs{m}(\abs{m}+2)\}$.

Suppose
\begin{align}\label{for:13}
 \abs{b_{n,n-2r}}\leq c_{n-2r}, \qquad 2\leq r\leq k-1,
\end{align}
where $n-2(k-1)>\max\{7,\,2\abs{m}(\abs{m}+2)\}$.

If $n-2k>\max\{7,\,2\abs{m}(\abs{m}+2)\}$ we have
\begin{align*}
 \abs{b_{n,n-2k}}&\leq \text{\small $\frac{2\abs{m}}{n-2k+1+\nu}$}\abs{b_{n,n-2k+2}}+\text{\small $\frac{n-2k+3-\nu}{n-2k+1+\nu}$}\abs{b_{n,n-2k+4}}\\
 &\overset{\text{(1)}}{\leq} \text{\small $\frac{2\abs{m}}{n-2k+1+\nu}$}\abs{c_{n-2k+2}}+\text{\small $\frac{n-2k+3-\nu}{n-2k+1+\nu}$}\abs{c_{n-2k+4}}\\
 &=\text{\small $\frac{2\abs{m}\cdot 4}{(n-2k+1+\nu)(n-2k+3+\nu)}$}\Pi_{\ell=1}^{k-1}\big(1+\text{\small $\frac{\abs{m}+2-\nu}{n-2\ell-1+\nu}$}\big)\\
 &+\text{\small $\frac{4(n-2k+3-\nu)}{(n-2k+1+\nu)(n-2k+5+\nu)}$}\Pi_{\ell=1}^{k-2}\big(1+\text{\small $\frac{\abs{m}+2-\nu}{n-2\ell-1+\nu}$}\big)\\
 &= \text{\small $\frac{4}{n-2k+1+\nu}$}\Pi_{\ell=3}^{k-2}\big(1+\text{\small $\frac{\abs{m}+2-\nu}{n-2\ell-1+\nu}$}\big)\\
 &\Big(\text{\small $\frac{2\abs{m}}{n-2k+3+\nu}$}\big(1+\text{\small $\frac{m+2-\nu}{n-2k+1+\nu}$}\big)+\text{\small $\frac{n-2k+3-\nu}{n-2k+5+\nu}$}\Big).
\end{align*}
In $(1)$ we use the assumption \eqref{for:13}.

Then to prove that $\abs{b_{n,n-2k}}\leq c_{n-2k}$ it suffices to show that
\begin{align}\label{for:10}
&\text{\small $\frac{2\abs{m}}{n-2k+3+\nu}$}\big(1+\text{\small $\frac{\abs{m}+2-\nu}{n-2k+1+\nu}$}\big)+\text{\small $\frac{n-2k+3-\nu}{n-2k+5+\nu}$}\notag\\
&\leq \Pi_{\ell=k-1}^{k}\big(1+\text{\small $\frac{\abs{m}+2-\nu}{n-2\ell-1+\nu}$}\big).
\end{align}
If $2-\abs{m}\geq\nu$ we have
\begin{align*}
 &\text{\small $\frac{n-2k+3-\nu}{n-2k+5+\nu}$}\\
 &=1-\text{\small $\frac{2\nu+2}{n-2k+5+\nu}$}\leq 1-\text{\small $\frac{2\nu-4}{n-2k-1+\nu}$}\\
 &\overset{\text{(1)}}{\leq} \big(1+\text{\small $\frac{\abs{m}+2-\nu}{n-2k-1+\nu}$}\big)\big(1+\text{\small $\frac{\abs{m}+2-\nu}{n-2k-1+\nu}$}-\text{\small $\frac{2\abs{m}}{n-2k-1+\nu}$}\big)\\
 &\leq\big(1+\text{\small $\frac{\abs{m}+2-\nu}{n-2k+1+\nu}$}\big)\big(1+\text{\small $\frac{\abs{m}+2-\nu}{n-2k-1+\nu}$}-
 \text{\small $\frac{2\abs{m}}{n-2k+3+\nu}$}\big).
\end{align*}
Here in $(1)$ we use the relation: $(1+a)(1+b)\geq 1+a+b$ if $ab>0$. Hence we get \eqref{for:10}.

If $\nu>2-\abs{m}$, we have
\begin{align*}
 &1-\text{\small $\frac{2\nu+2}{n-2k+5+\nu}$}\\
 &=1-\text{\small $\frac{2\nu-4}{n-2k-1+\nu}$}+\big(\text{\small $\frac{2\nu-4}{n-2k-1+\nu}$}-\text{\small $\frac{2\nu+2}{n-2k+5+\nu}$}\big)\\
 &=1-\text{\small $\frac{2\nu-4}{n-2k-1+\nu}$}+\text{\small $\frac{-6(n-2k-\nu+3)}{(n-2k-1+\nu)(n-2k+5+\nu)}$}\\
 &\overset{\text{(1)}}{\leq} 1-\text{\small $\frac{2\nu-4}{n-2k-1+\nu}$}+\text{\small $\frac{-(\abs{m}+2-\nu)(\abs{m}-2+\nu)}{(n-2k-1+\nu)^2}$}
 \end{align*}
 \begin{align*}
  &= \big(1+\text{\small $\frac{\abs{m}+2-\nu}{n-2k-1+\nu}$}\big)\big(1+\text{\small $\frac{\abs{m}+2-\nu}{n-2k-1+\nu}$}-\text{\small $\frac{2\abs{m}}{n-2k-1+\nu}$}\big)\\
 &\leq\big(1+\text{\small $\frac{\abs{m}+2-\nu}{n-2k+1+\nu}$}\big)\big(1+\text{\small $\frac{\abs{m}+2-\nu}{n-2k-1+\nu}$}-
 \text{\small $\frac{2\abs{m}}{n-2k+3+\nu}$}\big).
\end{align*}
Here in $(1)$ we use that
\begin{align*}
 n-2k+5+\nu\leq 2(n-2k-1+\nu),
\end{align*}
if $n-2k-\nu>7$; and
\begin{align*}
 n-2k-\nu>(\abs{m}+2-\nu)(\abs{m}-2+\nu),
\end{align*}
if $n-2k-\nu>2|m|(|m|+2)$. Hence we finish the proof.
\end{proof}

\begin{corollary}\label{le:1}
Suppose $\nu\in (-1,1)\backslash 0$. If $n\geq0$ and $n-2k\geq 0$ then
\begin{align*}
 \abs{b_{n,n-2k}}\text{\small$\frac{\norm{u_{n-2k}}}{\norm{u_n}}$}\leq C_mn ^{\frac{\abs{m}+2}{2}}.
\end{align*}
\end{corollary}
\begin{proof}
Similar to the proof of Corollary \ref{le:8},  by using $\log(1+x)\leq x$ and integral inequalities we have
\begin{align*}
 \abs{c_{n-2k}}&\leq C_m\big(\text{\small $\frac{n-1+\nu}{n-2k-1+\nu}$}\big)^{\frac{\abs{m}+2-\nu}{2}}.
\end{align*}
By Lemma \ref{le:10} we have
\begin{align*}
 \text{\small$\frac{\norm{u_{n-2k}}}{\norm{u_n}}$}\leq C\big(\text{\small$\frac{n-2k+1}{n+1}$}\big)^{-\frac{\nu}{2}}.
\end{align*}
The above estimates and Lemma \ref{le:6} imply the conclusion.
\end{proof}
\subsection{Upperbounds for $\nu\in \ZZ^+$} Let
\begin{align*}
c_{n-2k}=\text{\small $\frac{2}{[(n-2k+1)^2-\nu^2]^{\frac{1}{2}}}$}\Pi_{\ell=1}^{k}(1+\text{\small $\frac{\abs{m}}{n-2\ell-1-\nu}$}).
\end{align*}
\begin{lemma}\label{le:9}
Suppose $\nu\in \ZZ^+$ and $n-2k-\nu>\max\{6,\,2|m|^2+2\}$. Then
\begin{align*}
  \abs{b_{n,n-2k}}\text{\small$\frac{\norm{u_{n-2k}}}{\norm{u_n}}$}\leq c_{n-2k}.
\end{align*}
\end{lemma}
\begin{proof}
We prove by induction. By \eqref{for:18} and \eqref{for:12} it is easy to check that
\begin{align*}
 \abs{b_{n,n-2k}}\text{\small$\frac{\norm{u_{n-2k}}}{\norm{u_n}}$}\leq c_{n-2k},\qquad k=1,\,2.
\end{align*}
if $n-2k-\nu>\max\{7,\,2\abs{m}(\abs{m}+2)\}$.

Suppose
\begin{align}\label{for:14}
 \abs{b_{n,n-2r}}\text{\small$\frac{\norm{u_{n-2r}}}{\norm{u_n}}$}\leq c_{n-2r}, \qquad 1\leq r\leq k-1,
\end{align}
where $n-2k-\nu>\max\{7,\,2\abs{m}(\abs{m}+2)\}$. Then
\begin{align*}
&\abs{b_{n,n-2k}}\text{\small$\frac{\norm{u_{n-2k}}}{\norm{u_n}}$}\\
&\leq \text{\small $\frac{2\abs{m}}{n-2k+1+\nu}$}\abs{b_{n,n-2k+2}}\text{\small$\frac{\norm{u_{n-2k}}}{\norm{u_n}}$}\\
 &+\text{\small $\frac{n-2k+3-\nu}{n-2k+1+\nu}$}\abs{b_{n,n-2k+4}}\text{\small$\frac{\norm{u_{n-2k}}}{\norm{u_n}}$}\\
 &\overset{\text{(1)}}{=} \text{\small $\frac{2\abs{m}}{n-2k+1+\nu}$}\big(\text{\small$\frac{n-2k+1+\nu}{n-2k+1-\nu}$}\big)^{\frac{1}{2}}\abs{b_{n,n-2k+2}}\text{\small$\frac{\norm{u_{n-2k+2}}}{\norm{u_n}}$}\\
 &+\text{\small $\frac{n-2k+3-\nu}{n-2k+1+\nu}$}\big(\text{\small$\frac{n-2k+1+\nu}{n-2k+1-\nu}\cdot \frac{n-2k+3+\nu}{n-2k+3-\nu}$}\big)^{\frac{1}{2}}\abs{b_{n-2k+4}}\text{\small$\frac{\norm{u_{n-2k+4}}}{\norm{u_n}}$}\\
 &\overset{\text{(2)}}{\leq}\text{\small $\frac{2\abs{m}}{n-2k+1+\nu}$}\big(\text{\small$\frac{n-2k+1+\nu}{n-2k+1-\nu}$}\big)^{\frac{1}{2}}\abs{c_{n-2k+2}}\\
 &+\text{\small $\frac{n-2k+3-\nu}{n-2k+1+\nu}$}\big(\text{\small$\frac{n-2k+1+\nu}{n-2k+1-\nu}\cdot \frac{n-2k+3+\nu}{n-2k+3-\nu}$}\big)^{\frac{1}{2}}\abs{c_{n-2k+4}}\\
 &= \text{\small $\frac{2}{[(n-2k+1)^2-\nu^2]^{\frac{1}{2}}}$}\Pi_{\ell=1}^{k-2}\big(1+\text{\small $\frac{\abs{m}}{n-2\ell-1-\nu}$}\big)\\
 &\cdot\Big(\text{\small $\frac{2\abs{m}}{[(n-2k+3)^2-\nu^2]^{\frac{1}{2}}}$}\big(1+\text{\small $\frac{m}{n-2k+1-\nu}$}\big)\\
 &+\text{\small $\frac{[(n-2k+3)^2-\nu^2]^{\frac{1}{2}}}{[(n-2k+5)^2-\nu^2]^{\frac{1}{2}}}$}\Big).
\end{align*}
In $(1)$ we use \eqref{for:18}; in $(2)$ we use the assumption \eqref{for:14}.

Then to prove that $\abs{b_{n,n-2k}}\text{\small$\frac{\norm{u_{n-2k}}}{\norm{u_n}}$}\leq c_{n-2k}$ it suffices to show that
\begin{align}\label{for:5}
&\text{\small $\frac{2\abs{m}}{[(n-2k+3)^2-\nu^2]^{\frac{1}{2}}}$}\big(1+\text{\small $\frac{\abs{m}}{n-2k+1-\nu}$}\big)+\text{\small $\frac{[(n-2k+3)^2-\nu^2]^{\frac{1}{2}}}{[(n-2k+5)^2-\nu^2]^{\frac{1}{2}}}$}\notag\\
&\leq \Pi_{\ell=k-1}^{k}\big(1+\text{\small $\frac{\abs{m}}{n-2\ell-1-\nu}$}\big).
\end{align}
For any $b\in\RR$ and $x\geq \max\{2|m|^2+2, \abs{b}+6\}$ and  we have
\begin{align*}
 \text{\small $\frac{\sqrt{x^2-b^2}}{\sqrt{(x+2)^2-b^2}}$}&=\text{\small $\sqrt{1-\frac{4x+4}{(x+2)^2-b^2}}$}\overset{\text{(1)}}{\leq} 1-\text{\small $\frac{2x+2}{(x+2)^2-b^2}$}\\
 &\overset{\text{(2)}}{\leq} 1-\text{\small $\frac{m^2}{(x+2)^2-3(2(x+2)-m^2)-b^2}$}\overset{\text{(3)}}{\leq} 1-\text{\small $\frac{|m|^2}{x^2-b^2}$}\\
 &= \big(1+\text{\small $\frac{\abs{m}}{\sqrt{x^2-b^2}}$}\big)\big(1+\text{\small $\frac{\abs{m}}{\sqrt{x^2-b^2}}$}-\text{\small $\frac{2\abs{m}}{\sqrt{x^2-b^2}}$}\big)\\
  &\overset{\text{(4)}}{\leq} \big(1+\text{\small $\frac{\abs{m}}{x-b}$}\big)\big(1+\text{\small $\frac{\abs{m}}{x-b}$}-\text{\small $\frac{2\abs{m}}{\sqrt{x^2-b^2}}$}\big)\\
 &\leq \big(1+\text{\small $\frac{\abs{m}}{x-2-b}$}\big)\big(1+\text{\small $\frac{\abs{m}}{x-4-b}$}-\text{\small $\frac{2\abs{m}}{\sqrt{x^2-b^2}}$}\big).
\end{align*}
In $(1)$ we use the inequality $\sqrt{1-y}\leq 1-\frac{1}{2}y$ if $0\leq y\leq 1$; in $(2)$ we use the fact that
\begin{align*}
  \text{\small $\frac{2x+2}{(x+2)^2-b^2}$}<\frac{1}{3}\quad\text{ and }\quad (x+2)^2-3(2(x+2)-m^2)-b^2>0;
\end{align*}
if $x\geq \abs{b}+6$; in $(3)$ we use the fact that if $x>2m^2$ then
\begin{align*}
 (x+2)^2-3(2(x+2)-m^2)<x^2;
\end{align*}
in $(4)$ we use that $\frac{1}{\sqrt{x^2-b^2}}\leq \frac{1}{x-b}$ if $x>b>0$.

Then \eqref{for:5} follows immediately from the above inequality by letting $b=\nu$ and $x=n-2k+3$. Hence we finish the proof.
\end{proof}
\begin{corollary}\label{le:2}
Suppose $\nu\in \ZZ^+$. If $n\geq 1+\nu$ and $n-2k\geq 1+\nu$ then
\begin{align*}
 \abs{b_{n,n-2k}}\text{\small$\frac{\norm{u_{n-2k}}}{\norm{u_n}}$}\leq C_mn ^{\frac{\abs{m}}{2}}.
\end{align*}
\end{corollary}
\begin{proof}
Similar to the proof of Corollary \ref{le:8}, by using $\log(1+x)\leq x$ and integral inequalities we have
\begin{align*}
 \abs{c_{n-2k}}\leq C_m(\text{\small$\frac{n-\nu}{n-2k-\nu}$})^{\frac{\abs{m}}{2}},
\end{align*}
if $1\leq k\leq \frac{1}{2}(n-1-\nu)$. This and Lemma \ref{le:9} imply the conclusion.
\end{proof}
It is clear that Proposition \ref{po:1} is a direct consequence of Corollary \ref{le:8}, \ref{le:1} and \ref{le:2}.
\section{Invariant distributions and Sobolev norms of the solution}
\begin{proposition}\label{le:11}
For any $n\in\ZZ\backslash 0$ and $t\geq 0$ we have
\begin{align*}
 \norm{f_{\{n\}}}_t\leq C_m(1+\mu+2n^2)^\frac{t}{2}|n| ^{\frac{\abs{m}+3}{2}}\norm{u_n}.
\end{align*}
\end{proposition}
\begin{proof}
By \eqref{for:9} and \eqref{for:19} we have
\begin{align*}
  \norm{f_{\{n\}}}_t^2&\leq\sum_{\substack{2\leq 2k\leq |n|\\ |n|-2k\in I_\nu}}(1+\mu+2(|n|-2k)^2)^t |b_{|n|,|n|-2k}|^2 \norm{u_{|n|-2k}}^2\\
  &\overset{\text{(1)}}{\leq} \sum_{\substack{2\leq 2k\leq |n|\\ n-2k\in I_\nu}} C_m(1+\mu+2n^2)^t|n| ^{|m|+2}\norm{u_n}^2\\
  &\leq C_m(1+\mu+2n^2)^t|n| ^{|m|+3}\norm{u_n}^2
\end{align*}
Here $(1)$ is a direct consequence of Proposition \ref{po:1}. Then we finish the proof.
\end{proof}
\subsection{Invariant distributions in $\mathcal{H}_\nu$}\label{sec:9} For $n\in \mathcal{S}^\delta_\nu$ the linear functional
\begin{align*}
  \text{\small$\mathcal{D}_{\nu,n}^{\delta,m}: g=\sum_{k\in I_\nu}g_ku_k$}&\to \text{\small$-g_n+\text{\small$\frac{\nu-n-1}{2}$}\sum_{j\in I_\nu\backslash \mathcal{S}^\delta_\nu} b_{j,n+2}g_j$}\\
  &+\text{\small$m\sum_{j\in I_\nu\backslash \mathcal{S}^\delta_\nu} b_{j,n}g_j
  +\text{\small$\frac{\nu+n-1}{2}$}\sum_{j\in I_\nu\backslash \mathcal{S}^\delta_\nu} b_{j,n-2}g_j$}
\end{align*}
is defined for any $g\in\mathcal{H}^{\frac{\abs{m}+8}{2}}_\nu$. We note that if $\nu\in \ZZ^+$, then $b_{\nu+1,\nu-1}=0$.
In fact, from \eqref{for:43} we have
\begin{align}\label{for:3}
 g+\sum_{n\in \mathcal{S}^\delta_\nu}\text{\small$\mathcal{D}_{\nu,n}^{\delta,m}$}(g)u_n=(X+m)\big(\sum_{n\in I_\nu\backslash \mathcal{S}^\delta_\nu}g_nf_{\{n\}}\big).
\end{align}

\begin{lemma}\label{le:3} Suppose $g\in\mathcal{H}^{s}_\nu$, $s\geq\frac{\abs{m}+8}{2}$. Then
\begin{align*}
 \norm{\mathcal{D}_{\nu,n}^{\delta,m}(g)u_n}_t\leq C_m\norm{g}_{\frac{\abs{m}+8}{2}+t}, \qquad t\leq s-\text{\small$\frac{\abs{m}+8}{2}$}.
\end{align*}

\end{lemma}
\begin{proof}
We have
\begin{align*}
  &\text{\small$\norm{\mathcal{D}_{\nu,n}^{\delta,m}(g)u_n}_t$}\\
  &\overset{\text{(1)}}{\leq} C(1+u+2n^2)^{\frac{t}{2}}\text{\small$(\abs{m}+\abs{\nu}) \sum_{\ell=-2,0,2}\sum_{j\in I_\nu\backslash \mathcal{S}^\delta_\nu} |b_{j,n+\ell}g_j|\norm{u_n}$}\\
  &\leq
  C(1+u)^{\frac{t}{2}}\text{\small$(\abs{m}+\abs{\nu}) \sum_{\ell=-2,0,2}\sum_{j\in I_\nu\backslash \mathcal{S}^\delta_\nu} \big(|b_{j,n+\ell}|\text{\small$\frac{\norm{u_{n+\ell}}}{\norm{u_j}}$}$}\big)
  \text{\small$\frac{\norm{u_{n}}}{\norm{u_{n+\ell}}}$}(\abs{g_j}\norm{u_j})\\
  &\overset{\text{(2)}}{\leq} C_m(1+u)^{\frac{t}{2}}\text{\small$(\abs{m}+\abs{\nu}) \sum_{\ell=-2,0,2}\sum_{j\in I_\nu\backslash \mathcal{S}^\delta_\nu} $}
  \text{\small$\frac{\norm{u_{n}}}{\norm{u_{n+\ell}}}$}(\abs{j}^{\frac{\abs{m}+2}{2}}\abs{g_j}\norm{u_j})\\
  &\overset{\text{(3)}}{\leq} C_m(1+u)^{\frac{t}{2}}(\abs{\nu}+1)^2\text{\small$\sum_{j\in I_\nu} $}
  \abs{j}^{\frac{\abs{m}+2}{2}}\abs{g_j}\norm{u_j}\\
  &\leq C_m(1+u)^{\frac{t}{2}}(\abs{\nu}+1)^2\big(\text{\small$\sum_{j\in I_\nu} $}
  \abs{j}^{\abs{m}+4}\abs{g_j}^2\norm{u_j}^2\big)^{\frac{1}{2}}\\
  &\leq C_m\big(\text{\small$\sum_{j\in I_\nu} $}
  \abs{1+u+2j^2}^{\frac{\abs{m}+4}{2}+t+2}\abs{g_j}^2\norm{u_j}^2\big)^{\frac{1}{2}}\\
  &\overset{\text{(4)}}{\leq}\text{\small$C_m\norm{g}_{\frac{\abs{m}+8}{2}+t}$}
\end{align*}
$(2)$ is a direct consequence of Proposition \ref{po:1}; in $(3)$ we use \eqref{for:18}; in $(1)$ and $(4)$ we use \eqref{for:9}. Hence we finish the proof.
\end{proof}
In fact, these $\text{\small$\mathcal{D}_{\nu,n}^{\delta,m}(g)$}$ are $(X-m)$-invariant distributions.
\begin{proposition}\label{po:2}
If there is $f\in \mathcal{H}^{\frac{\abs{m}+10}{2}}_\nu$ such that $(X+m)f=g$, then for $n\in \mathcal{S}^\delta_\nu$, $\text{\small$\mathcal{D}_{\nu,n}^{\delta,m}(g)$}=0$.
\end{proposition}
We postpone the proof to Section \ref{sec:3}.

\subsection{Sobolev estimates of the solution}\label{sec:5} We are now ready to give the explicit solution of the twisted equation.
\begin{lemma}\label{le:19}
Suppose $g\in (\mathcal{H}^\delta_\nu)^{s}$ with $s>\frac{\abs{m}+8}{2}$. If $\mathcal{D}_{\nu,n}^{\delta,m}(g)=0$ for any $n\in \mathcal{S}^\delta_\nu$, then the equation \eqref{for:105}
has a solution $f\in (\mathcal{H}^\delta_\nu)^{s-\frac{\abs{m}}{2}-3}$ with estimates
\begin{align}\label{for:20}
\norm{f}_t\leq C_m \norm{g}_{t+\frac{\abs{m}}{2}+3},\qquad t\leq s-\text{\small$\frac{\abs{m}}{2}$}-3.
\end{align}
Furthermore, we have
\begin{align}\label{for:21}
 &\norm{f|_n}_t\leq\left\{\begin{aligned}&C_{m}\norm{\Theta^{\frac{\abs{m}+5}{2}}(g|_{n-2})}_{t+\frac{1}{2}}; &\quad &\text{if }n\leq0\\
&C_{m}\norm{\Theta^{\frac{\abs{m}+5}{2}}(g|_{n+2})}_{t+\frac{1}{2}},&\quad& \text{if }n> 0,
\end{aligned}
 \right.
\end{align}
for any $t\leq s-\text{\small$\frac{\abs{m}}{2}$}-3$ (see \eqref{for:107}).

\end{lemma}
\begin{proof} Let \begin{align}\label{for:48}
  f=\sum_{n\in I_\nu\backslash S_\nu}g_n f_{\{n\}}
\end{align}
(see \eqref{for:44}). From \eqref{for:3} we see that the assumption $\mathcal{D}_{\nu,n}^{\delta,m}(g)=0$ for any $n\in \mathcal{S}^\delta_\nu$ implies that $f$ is a formal solution of the equation \eqref{for:105}.

Furthermore, from \eqref{for:44} we have
\begin{align*}
 &f|_n=\left\{\begin{aligned}&\sum_{\substack{n\in I_\nu\backslash S_\nu,\\k\leq-1}} g_{n+2k}f_{\{n+2k\}}; &\quad &\text{if }n\leq0\\
&\sum_{\substack{n\in I_\nu\backslash S_\nu,\\k\geq1}}g_{n+2k}f_{\{n+2k\}},&\quad& \text{if }n> 0.
\end{aligned}
 \right.
\end{align*}
By Proposition \ref{le:11} for any $n\in\NN$, any $\delta>0$ and $t<s-\frac{\abs{m}+5}{2}$ we have
\begin{align}\label{for:108}
 \norm{f|_n}_t&=\sum_{\substack{n\in I_\nu\backslash S_\nu,\\k\geq1}}|g_{n+2k}| \norm{f_{\{n+2k\}}}_t\notag\\
 &\leq \sum_{\substack{n+2k\in I_\nu,\\k\geq1}}C_m\big(1+u+2(n+2k)^2\big)^\frac{t}{2}(n+2k)^{\frac{\abs{m}+3}{2}}\norm{g_{n+2k}u_{n+2k}}\notag\\
 &\overset{\text{(1)}}{\leq} \sum_{k\geq1}C_{m,\delta}\big(1+u+2(n+2k)^2\big)^\frac{t}{2}(n+2k)^{\frac{\abs{m}+3}{2}}\notag\\
 &\cdot\text{\small$\frac{\norm{\Theta^{\frac{\abs{m}+5}{2}}(g|_{n+2})}_{t+\delta}}{(1+u+2(n+2k)^2)^{\frac{t+\delta}{2}}(n+2k)^{\frac{\abs{m}+5}{2}}}$}\notag\\
 &\leq \sum_{k\geq1}C_{m,\delta}(n+2k)^{-1-\delta}\norm{\Theta^{\frac{\abs{m}+5}{2}}(g|_{n+2})}_{t+\delta}\notag\\
 &\leq C_{m,\delta}\norm{\Theta^{\frac{\abs{m}+5}{2}}(g|_{n+2})}_{t+\delta}.
\end{align}
Here in $(1)$ we use \eqref{for:9}.

Hence we prove \eqref{for:21} for $n>0$. The proof of the case of $n<0$ is similar. Then we get \eqref{for:21}. It is clear that \eqref{for:20} is a direct consequence of \eqref{for:21}.
Hence we finish the proof.
\end{proof}
\subsection{Proof of proposition \ref{po:2}}\label{sec:3}
We need an additional step to get to the proof.
\begin{lemma} We have
\begin{align*}
 \text{\small$\mathcal{D}_{\nu,n}^{\delta,m}\big((X+m)u_k\big)$}=0, \qquad \forall \,k\in I_\nu.
\end{align*}
\end{lemma}
\begin{proof}
Set $g=(X+m)u_k$. Let
\begin{align}\label{for:6}
 g'=g+\sum_{n\in \mathcal{S}^\delta_\nu}\text{\small$\mathcal{D}_{\nu,n}^{\delta,m}$}(g)u_n=(X+m)u_k+\sum_{n\in \mathcal{S}^\delta_\nu}\text{\small$\mathcal{D}_{\nu,n}^{\delta,m}$}(g)u_n.
\end{align}
Then $\text{\small$\mathcal{D}_{\nu,n}^{\delta,m}$}(g')=0$, for any $n\in \mathcal{S}^\delta_\nu$. Since $g'$ is $\Theta$-finite (see Definition \ref{de:1}), by Lemma \ref{le:19}, we have $f\in \mathcal{H}^\infty$ such that
\begin{align}\label{for:15}
 (X+m)f=g'
\end{align}
moreover, $f$ is also $\Theta$-finite.

Hence, it follows from \eqref{for:6} and \eqref{for:15} that
\begin{align*}
 \sum_{n\in \mathcal{S}^\delta_\nu}\text{\small$\mathcal{D}_{\nu,n}^{\delta,m}$}(g)u_n=(X+m)(f-u_k).
\end{align*}
Since $f-u_k$ is also $\Theta$-finite, the observation after Definition \ref{de:1} implies that $f=u_k$ and $\sum_{n\in \mathcal{S}^\delta_\nu}\text{\small$\mathcal{D}_{\nu,n}^{\delta,m}$}(g)u_n=0$.
Hence we finish the proof.

\end{proof}
Now we are ready to prove Proposition \ref{po:2}. Write $f=\sum_{k\in I_\nu}f_ku_k$. For any $\ell\in\NN$ let
\begin{align*}
 \mathfrak{g}_\ell=(X+m)(\sum_{k\in I_\nu, \abs{k}\leq \ell}f_ku_k).
\end{align*}
From \eqref{for:9} we see that
\begin{align*}
 \big\|\sum_{k\in I_\nu, \abs{k}> \ell}f_ku_k \big\|_{\frac{\abs{m}+10}{2}}\to 0,\qquad\text{as }\ell\to \infty.
\end{align*}
Hence we have
\begin{align*}
 \norm{\mathfrak{g}_\ell-g}_{\frac{\abs{m}+8}{2}}\to 0, \qquad\text{as }\ell\to \infty.
\end{align*}
Then it follows from Lemma \ref{le:3} that
\begin{align*}
\text{\small$\mathcal{D}_{\nu,n}^{\delta,m}(g)$}=\lim_{\ell\to\infty} \text{\small$\mathcal{D}_{\nu,n}^{\delta,m}(g_\ell)$}=0.
\end{align*}
Hence we finish the proof.

\subsection{Extended distributions and solutions}\label{sec:6}  For any $n\in \ZZ\backslash 0$ with $\abs{n}\geq \abs{\text{Re}(\nu)}+2$ we consider the subspace $\mathfrak{F}_{n,\nu}=\{g\in (\mathcal{H}_\nu)^s: g=g|_{n}\}$, $s\geq\frac{\abs{m}+8}{2}$. For any $g\in \mathfrak{F}_{n,\nu}$ let $\tilde{f}=(\sum_{\ell\in I_\nu} g_\ell f_{\{\ell\}})|_{n}$.  Then it follows from \eqref{for:108} of the proof of Lemma \ref{le:19} immediately that
 \begin{align}\label{for:109}
 \norm{\tilde{f}}_t\leq\norm{g}_{t+\frac{\abs{m}}{2}+3},\qquad t\leq s-\text{\small$\frac{\abs{m}}{2}$}-3.
 \end{align}
It is clear that
\begin{align*}
  (X+m)(\sum_\ell g_\ell f_{\{\ell\}})|_{n}=g-\tilde{g},
\end{align*}
where
\begin{align*}
 &\tilde{g}=\left\{\begin{aligned}&\text{\small$\mathcal{E}_{\nu,n,1}^{\delta,m}(g)$}u_n+\text{\small$\mathcal{E}_{\nu,n,2}^{\delta,m}(g)$}u_{n-2}; &\quad &\text{if }n>0\\
&\text{\small$\mathcal{E}_{\nu,n,1}^{\delta,m}(g)$}u_n+\text{\small$\mathcal{E}_{\nu,n,2}^{\delta,m}(g)$}u_{n+2},&\quad& \text{if }n<0,
\end{aligned}
 \right.
 \end{align*}
and
\begin{align*}
 &\text{\small$\mathcal{E}_{\nu,n,1}^{\delta,m}(g)$}=\left\{\begin{aligned}&\text{\small$-g_n
 +\text{\small$\frac{\nu-n-1}{2}$}\sum_{j\geq n+4} b_{j,n+2}g_j$}+\text{\small$m\sum_{j\geq n+2} b_{j,n}g_j$}; &\quad &\text{if }n>0\\
&\text{\small$-g_n
 +\text{\small$\frac{\nu+n-1}{2}$}\sum_{j\leq n-4} b_{j,n-2}g_j$}+\text{\small$m\sum_{j\leq n-2} b_{j,n}g_j$},&\quad& \text{if }n<0;
\end{aligned}
 \right.
 \end{align*}
and
\begin{align*}
 &\text{\small$\mathcal{E}_{\nu,n, 2}^{\delta,m}(g)$}=\left\{\begin{aligned}&\text{\small$\frac{\nu+n-1}{2}\sum_{j\geq n+2} b_{j,n}g_j$}; &\quad &\text{if }n>0\\
&\text{\small$\frac{\nu-n-1}{2}\sum_{j\leq n-2} b_{j,n}g_j$},&\quad& \text{if }n<0.
\end{aligned}
 \right.
 \end{align*}
From \eqref{for:109} we have
\begin{align*}
\norm{\tilde{g}}_t\leq \norm{g}_t+\norm{(X+m)\tilde{f}}_t\leq \norm{g}_t+C_m\norm{\tilde{f}}_{t+1}\leq C_m\norm{g}_{t+\frac{\abs{m}}{2}+4}
\end{align*}
if $t\leq s-\text{\small$\frac{\abs{m}}{2}$}-4$.
\section{Proof of Theorem \ref{th:1}, \ref{th:3} and \ref{th:7}}
\emph{Proof of Theorem \ref{th:3}}: Let
\begin{align*}
 g_1=g+\sum_{n\in \mathcal{S}^\delta_\nu}\mathcal{D}_{\nu,n}^{\delta,m}(g)u_n.
\end{align*}
 We see that $\mathcal{D}_{\nu,n}^{\delta,m}(g_1)=0$ for any $n\in \mathcal{S}^\delta_\nu$. Then by Lemma \ref{le:19},
the equation
\begin{align*}
 (X+m)f=g_1
\end{align*}
has a solution $f\in \mathcal{H}_\nu$, where $f=\sum_{n\in I_\nu\backslash S_\nu}g_n f_{\{n\}}$. Hence \eqref{for:111} follows from Lemma \ref{le:19}. \eqref{for:112} is a direct consequence of
Proposition \ref{po:2}. \eqref{for:113} is Lemma \ref{le:3}.

\medskip
\emph{Proof of Theorem \ref{th:7}}: It follows directly from Section \ref{sec:6}.

\medskip
\emph{Proof of Theorem \ref{th:1}}: \eqref{for:17} is a direct consequence of Lemma \ref{le:4}.

\smallskip
 We consider the decomposition of $\pi$ as in \eqref{for:1} and the Sobolev spaces decomposition as in \eqref{for:51}. We write $g=\int_{\oplus}g_\mu dS(\mu)$.

 \eqref{for:22}: Arguments in Section \ref{sec:17} allow us to apply
\eqref{for:111} of Theorem \ref{th:3} to each $g_\mu$. Hence the equation
\begin{align*}
 (X+m)f_\mu=g_\mu
\end{align*}
has a solution $f_\mu$ with estimates
\begin{align*}
\norm{f_\mu}_t\leq C_m \norm{g_\mu}_{t+\frac{\abs{m}}{2}+3}
\end{align*}
if $t\leq s-\text{\small$\frac{\abs{m}}{2}$}-3$.

Let $f=\int_{\oplus}f_\mu dS(\mu)$. Then
\begin{align}\label{for:24}
\norm{f}^2_t&=\int_{\oplus}\ell(\mu)\norm{f_\mu}^2 dS(\mu)\leq C_m \int_{\oplus}\ell(\mu)\norm{g_\mu}^2_{t+\frac{\abs{m}}{2}+3}dS(\mu)\notag\\
&=\norm{g}^2_{t+\frac{\abs{m}}{2}+3}
\end{align}
if $t\leq s-\text{\small$\frac{\abs{m}}{2}$}-3$. Hence we finish the proof.

\smallskip
\eqref{for:23}: By above arguments, we write $f=\int_{\oplus}f_\mu dS(\mu)$. The assumption implies that $f_\mu\in \mathcal{H}^{\frac{\abs{m}+10}{2}}_\mu$ for almost all $\mu$. Then it follows from \eqref{for:112} and \eqref{for:111} of Theorem \ref{th:3} that $f_\mu\in \mathcal{H}_\mu^{s-\frac{\abs{m}}{2}-3}$ with the estimate
\begin{align*}
\norm{f_\mu}_{s-\frac{\abs{m}}{2}-3}\leq C_m \norm{g_\mu}_{s}
\end{align*}
for almost all $\mu$.

Following the same way as in \eqref{for:24}, we have
\begin{align*}
\norm{f}_{s-\frac{\abs{m}}{2}-3}\leq C_m \norm{g}_{s}.
\end{align*}
Hence we finish the proof.

\end{document}